\documentclass[a4paper,12pt]{article}
\usepackage{makeidx}
\usepackage{latexsym}
\usepackage{graphicx}

\usepackage{enumerate}
\usepackage[centertags]{amsmath}
\usepackage{amsfonts}
\usepackage{amssymb}
\usepackage{amsthm}
\usepackage{newlfont}
\usepackage{fancyhdr}

\newtheorem{satz}{Theorem}

\newtheorem{lemma}[satz]{Lemma}
\newtheorem{theorem}[satz]{Theorem}

\newtheorem{proposition}[satz]{Proposition}

\newtheorem{notation}[satz]{Notation}

\newtheorem{definition}[satz]{Definition}

 \fussy
\begin{document}

\begin{center}
{\bf\large Maximal Subsemigroups containing a particular semigroup }
\end{center}

\begin{center}
J$\ddot{o}$rg Koppitz$^{\dag}$, Tiwadee Musunthia$^{\sharp}$\\
\end{center}
$$\begin{array}{cc}
\text{Potsdam University,}^{\dag} & \hspace*{1cm}\text{Department of Mathematics}^{\sharp}\\
\text{Institute of Mathematics}&\hspace*{1cm}\text{Silpakorn University}\\
\text{Am Neuen Palais, D-14415}&\hspace*{1cm}\text{Nakorn Pathom, Thailand 73000}\\
\text{ Potsdam, Germany}& \hspace*{1cm}\text{Centre of Excellence in Mathematics}\\
& \hspace*{1cm}\text{272 Rama VI Road, Bangkok}\\
& \hspace*{1cm}\text{10400, Thailand}\\
\text{e-mail: koppitz@rz.uni-potsdam.de} &
\hspace*{1cm}\text{e-mail: tiwadee$\_$m@hotmail.com}
\end{array}$$

\begin{abstract}
We study maximal subsemigroups of the monoid $T(X)$ of all full
transformations on the set $X=N$ of natural numbers containing a
given subsemigroup $W$ of $T(X)$ where each element of a given set
$U$ is a generator of $T(X)$ modulo $W$. This note continue the
study of maximal subsemigroups on the monoid of all full
transformations on a infinite set.
\end{abstract}

In this note, we want to continue the study of maximal subsemigroups
of the semigroup $T(X)$ of all full transformations on an infinite
set, in particular, for the case $X$ is countable. The maximal
subsemigroups of $T(X)$ containing the symmetric group $Sym(X)$ of
all bijective mappings on an infinite set are already known. They
were determined by L. Heindorf ($X$ is countable) and by M. Pinsker
(any infinite set $X$) characterizing maximal clones
(\cite{[5]},\cite{[8]}).

The setwise stabilizer of any finite set $Y\subseteq X$ under
$Sym(X)$ is a subgroup of $Sym(X)$. In \cite{[3]}, the authors
determine the maximal subsemigroups of $T(X)$ containing the setwise
stabilizer of any finite set $Y\subseteq X$ under $Sym(X)$. For a
finite partition of $X$, one can also consider the (almost)
stabilizer. They form subsemigroups of $Sym(X)$ and in \cite{[3]},
the maximal subsemigroups of $T(X)$ containing such a subgroup are
determined. Also in \cite{[3]}, the maximal subsemigroups containing
the stabilizer of any uniform ultrafilter on $X$, which forms also a
group, are determined.

In the present note, we consider a countable infinite set $X$ and
characterize for a given subsemigroup $W$ of $T(X)$ and a given set
$U\subseteq T(X)$, where any element of $U$ is a generator modulo
$W$ (see \cite{[6]}), the maximal subsemigroups of $T(X)$ containing
$W$. As a consequence of this result, we obtain all maximal
subsemigroups of $T(X)$ containing $T(X)\setminus S$, where $S$ is a
given maximal subsemigroup of $T(X)$ containing $Sym(X)$.

\begin{notation}
Let $M\subseteq P(T(X))$ and let $J(M)$ be the set of all
$A\subseteq T(X)$ with

$A\subseteq \bigcup M :=\{a|\exists m\in M\wedge a\in m\}$,

$A\cap m\neq \emptyset $ for all $m\in M$, and

$\forall \alpha \in A\exists m\in M$ with $A\cap m=\{\alpha \}$.
\end{notation}

\begin{definition}
Let $U\subseteq T(X)$ and $W\leq T(X)$. Then we put

$Gen(U):=\{A\subseteq T(X)\mid A$ is finite and $\left\langle
A\right\rangle \cap U\neq \emptyset \}$ and

$\mathcal{H}(U,W):=\{A\subseteq T(X)\setminus W\mid A\in
J(Gen(U))\}$.
\end{definition}

\begin{theorem}
Let $W\leq S\leq T(X)$ and $U\subseteq T(X)$ with $U\cap W=\emptyset
$ such that $\left\langle W,\alpha \right\rangle =T(X)$ for all
$\alpha \in U$. Then the following statements are
equivalent:\newline (i) $S$ is maximal.\newline (ii) There is a
$H\in \mathcal{H}(U,W)$ with $S=T(X)\setminus H$. \label{thm1}
\end{theorem}

\begin{proof}
$(i)\Rightarrow (ii)$: Assume that $S\cap H\neq\emptyset$ for all
$H\in\mathcal{H}(U,W)$. Then there is $A\in Gen(U)$ with
$A\subseteqq S$. Otherwise $A\nsubseteqq S$, i.e. $A\setminus
S\neq\emptyset$, for all $A\in Gen(U)$. We consider the set
$Gen(U)=\{A_i|i\in \mathbb{N}\}$. We put
$\overline{A}_i=A_i\setminus S$ for all $i\in\mathbb{N}$,

$H_0=\emptyset$ and $\overline{H}_0=\emptyset$;

$H_{i+1}=H_i$ if $\overline{A}_{i+1}\cap H_i\neq\emptyset$ or
$\exists j>i+1$ with
$\overline{A}_j\setminus\overline{H}_i\subseteqq \overline{A}_{i+1}$
then we define

$\overline{H}_{i+1}=\overline{H}_i$;

$H_{i+1}=H_i\cup\{a_{i+1}\}$ if $\overline{A}_{i+1}\cap
H_i=\emptyset$ and $\forall j>i+1$,
$\overline{A}_j\setminus\overline{H}_i\nsubseteqq
\overline{A}_{i+1}$ then we define

$\overline{H}_{i+1}=(\overline{H}_i\cup\overline{A}_{i+1})\setminus\{a_{i+1}\}$
for $a_{i+1}\in \overline{A}_{i+1}\setminus\overline{H}_i$.

We want to show that
$\overline{A}_{i+1}\setminus\overline{H}_i\neq\emptyset$ for all
$i\in\mathbb{N}$. Let us consider $i\in\mathbb{N}$. We know that
$\overline{A}_i=A_i\setminus S\neq\emptyset$ by assumption. If
$H_i=\emptyset$ then
$\overline{A}_{i+1}\setminus\overline{H}_i\neq\emptyset$. Assume
that $H_i\neq\emptyset$. Then there exists $k\in\mathbb{N}$ with
$k<i$ such that $H_k\cup\{a_{k+1}\}=H_{k+1}=H_i$. If
$\overline{A}_{i+1}\setminus\overline{H}_k\subseteqq\overline{A}_{k+1}$
then $\overline{H}_{k+1}=\overline{H}_k=\overline{H}_i$. This
implies that
$\overline{A}_{i+1}\setminus\overline{H}_i\neq\emptyset$. Then
$\overline{A}_j\setminus\overline{H}_k\nsubseteqq\overline{A}_{k+1}$
for all $j>k+1$. Thus
$\overline{A}_{i+1}\setminus\overline{H}_k\nsubseteqq\overline{A}_{k+1}$
and there exists $x\in\overline{A}_{i+1}$ but
$x\notin\overline{H}_k$ and $x\notin\overline{A}_{k+1}$, i.e.
$x\in\overline{A}_{i+1}$ but
$x\notin\overline{H}_k\cup\overline{A}_{k+1}$. Moreover, we have
$x\in\overline{A}_{i+1}$ but
$x\notin(\overline{H}_k\cup\overline{A}_{k+1})\setminus\{a_{k+1}\}=\overline{H}_i$.
This shows that $x\in\overline{A}_{i+1}\setminus\overline{H}_i$.
This completes the proof that
$\overline{A}_{i+1}\setminus\overline{H}_i\neq\emptyset$. Moreover,
we put $H:=\underset{i\in \mathbb{N}}\bigcup H_i$.

We want to show that $H:=\underset{i\in \mathbb{N}}\bigcup H_i\in
J(Gen(U))$. First, we will show that $H\subseteqq \bigcup
Gen(U):=\{a_i|\exists A_i\in Gen(U)\wedge a_i\in A_i\}$. By
definition of $H$, we know that $H_0=\emptyset$, $H_{i+1}=H_i$ or
$H_{i+1}=H_i\cup\{a_{i+1}\}$ for $i\in\mathbb{N}$. This shows that
$H\subseteqq\bigcup Gen(U)$.

Let $i \in \mathbb{N}$. Then we have to show that $H\cap
A_i\neq\emptyset$. If $a_i \in H$ then all is clear. Otherwise,
assume that $a_i\notin H$. Then, $\overline{A}_i\cap
H_{i-1}\neq\emptyset$ and thus $\overline{A}_i\cap H_i\neq\emptyset$
or $\exists k>i$ with
$\overline{A}_k\setminus\overline{H}_{i-1}\subseteqq
\overline{A}_{i}$. We have to consider $\overline{A}_k$ with $k>i$.
Then $a_k \in H$ or we have the same cases as for $A_i$. Since the
elements of $Gen(U)$ are finite, this procedure finish. So we obtain
an $s \geqq i$ such that $a_s \in H$ and it is routine to see that
also $a_s \in \overline{A}_i$. Thus $H\cap A_i\neq\emptyset$.

Let $a \in H$. Then there exists an $i \in \mathbb{N}$ such that
$a=a_i \in A_i$. We have to show that for all $k \in \mathbb{N}
\setminus \{i\}$, if $a_k \in H$ (i.e. $a_k \in A_k$) then $a_k$
does not belong to $A_i$. We have two cases:

{\it Case 1:} $k>i$. Then $a_k$ does not belong to $H_i \subseteqq
H_{k-1}$ and moreover, it does not belong to $\overline{H}_i
\subseteqq\overline{H}_{k-1}$. But $\overline{A}_i\subseteq
H_i\cup\overline{H}_i$. So $a_k$ does not belong to $A_i$.

{\it Case 2:} $k<i$. Then $a_k \in H_{i-1}$ where $H_{i-1}\cap
A_i=\emptyset$, whence $a_k\notin A_i$.

This shows that $H\in J(Gen(U))$. But $H\cap S=\emptyset$ is a
contradiction. Then there is $A\in Gen(U)$ with $A\subseteqq S$,
i.e. $\langle A\rangle\cap U\neq\emptyset$. Let $\alpha\in\langle
A\rangle\cap U\subseteqq S$. Then $T(X)=\langle W,\alpha\rangle\leq
S$, i.e. $S=T(X)$, a contradiction. Hence, there is
$H\in\mathcal{H}(U,W)$ with $S\cap H=\emptyset$, i.e. $S\subseteqq
T(X)\setminus H$. We want to show that $T(X)\setminus H$ is a
semigroup. Let $\alpha,\beta\in T(X)\setminus H$. Assume that
$\alpha\beta\in H$. Then there is $A\in Gen(U)$ with $A\cap
H=\{\alpha\beta\}$. Since $U\cap\langle A\rangle\neq\emptyset$,
$U\cap\langle A\setminus\{\alpha\beta\}\cup
\{\alpha,\beta\}\rangle\neq\emptyset$, i.e.
$(A\setminus\{\alpha\beta\})\cup\{\alpha,\beta\}\in Gen(U)$. Since
$A\cap H=\{\alpha\beta\}$, $(A\setminus\{\alpha \beta \})\cup
\{\alpha ,\beta \}\in Gen(U)$, implies $\alpha\in H$ or $\beta\in
H$, a contradiction. Hence $T(X)\setminus H$ is a semigroup. Since
$S$ is a maximal subsemigroup of $T(X)$ and $S\subseteqq
T(X)\setminus H$ this implies $S=T(X)\setminus H$.

$(ii)\Rightarrow (i)$: Let $H\in\mathcal{H}(U,W)$ with
$S=T(X)\setminus H$. We have shown that $T(X)\setminus H$ is a
semigroup. Now, we want to show that it is a maximal subsemigrop of
$T(X)$. Let $\alpha\in H$. Then there is $A\in Gen(U)$ with $A\cap
H=\{\alpha\}$ and $T(X)=\langle W,A\rangle\subseteqq\langle
T(X)\setminus H,A\rangle=\langle T(X)\setminus H,\alpha\rangle$
since $\langle A\rangle\cap U\neq\emptyset$ and $\langle
W,\beta\rangle=T(X)$ for all $\beta\in U$. So, $\langle
T(X)\setminus H,\alpha\rangle=T(X)$. This shows that $T(X)\setminus
H$ is a maximal subsemigroup of $T(X)$.
\end{proof}

If $\alpha\in T(X)$ and $A\subseteqq X$ such that the restriction of
$\alpha$ to $A$ is injective and have the same range as $\alpha$,
then we will refer $A$ as transversal of $\alpha$ ($ker\alpha$
denotes the kernel of $\alpha$). We will also write $A\#\ker \alpha
$ if $A$ is a transversal of $\alpha $.

Let $D(\alpha):=X\setminus im\alpha$ $(im\alpha$ denotes the range
of $\alpha)$. The rank $\alpha $, i.e. the cardinality of $%
im\alpha $, is denoted by $rank(\alpha ):=\left\vert im\alpha
\right\vert $. Then $d(\alpha):=|D(\alpha)|$ is called defect of
$\alpha$ and $c(\alpha):=\underset{y\in im\alpha }{\sum }(\left\vert
y\alpha ^{-1}\right\vert -1)$ is called collapse of $\alpha$.

Moreover, we put $K(\alpha):=\{x\in
im\alpha||x\alpha^{-1}|=\aleph_0\}$ and $k(\alpha):=|K(\alpha)|$ is
called infinite contractive index. It is well known that
$d(\alpha\beta)\leq d(\alpha)+d(\beta)$, $k(\alpha\beta)\leq
k(\alpha)+k(\beta)$ \cite{[4]} and $c(\alpha\beta)\leq
c(\alpha)+c(\beta)$ \cite{[1]} for $\alpha,\beta\in T(X)$. For more
background in the theory of transformation semigroups see \cite{[4]}
and \cite{[7]}.

Now we want to determine the maximal subsemigroups of $T(X)$
containing $T(X)\setminus S$, where $S$ is one of the five maximal
subsemigroups of $T(X)$ containing $Sym(X)$. Let us introduce the
following five sets:

\begin{itemize}
\item $Inj(X)$ $:=\{\alpha \in T(X)\mid rank(\alpha )=\aleph _{0}, c(\alpha )=0$ and $d(\alpha )\neq 0\}
$ (the set of injective but not surjective mappings on $X$).

\item $Sur(X):=\{\alpha \in T(X)\mid rank(\alpha )=\aleph _{0}, c(\alpha )\neq 0$ and $d(\alpha )=0\}$
(the set of surjective but not injective mappings on $X$).

\item $C_{p}(X):=\{\alpha \in T(X)\mid rank(\alpha )=\aleph _{0}, k(\alpha )=\aleph _{0}\}$.

\item $IF(X):=\{\alpha \in T(X)\mid rank(\alpha )=\aleph _{0}, c(\alpha )=\aleph _{0}$ and $d(\alpha
)<\aleph _{0}\}$.

\item $FI(X):=\{\alpha \in T(X)\mid rank(\alpha )=\aleph _{0}, d(\alpha )=\aleph _{0}$ and $c(\alpha
)<\aleph _{0}\}$.
\end{itemize}

In \cite{[5]}, the following proposition was proved. Note that we
independently proved this proposition whilst of the work of L.
Heindorf. We thank Martin Goldstern for bringing these reference to
our consideration at the AAA82 in Potsdam (June 2011). For the sake
of completeness, we include the proof of this proposition.

\begin{proposition}
The following semigroups of $T(X)$ are maximal:
\[
T(X)\setminus H
\]%
for $H\in \{Inj(X),Sur(X),C_{p}(X),IF(X),FI(X)\}$. \label{prop1}
\end{proposition}

\begin{proof}
1) Let $\alpha ,\beta \in T(X)\setminus Inj(X)$. Assume that $\alpha
\beta \in Inj(X)$. Then $c(\alpha )=0$, i.e. $\alpha $ is injective.
Since $\alpha \notin Inj(X)$, $\alpha \in Sym(X)$. But $c(\alpha
\beta)=0$ and $\alpha \in Sym(X)$ implies $\beta $ is injective.
Since $\beta \notin Inj(X)$, $\beta \in Sym(X)$. So $\alpha \beta
\in Sym(X)$, i.e. $\alpha \beta $ is surjective, contradicts $\alpha
\beta \in Inj(X)$. This shows that $T(X)\setminus Inj(X)$ is a
semigroup. \newline Let $\alpha \in Inj(X)$. Then we will show that
$\left\langle T(X)\setminus
Inj(X),\alpha \right\rangle =T(X)$. For this let $\beta \in Inj(X)$. Let $%
a\in im\beta $. Let $\gamma \in T(X)$ with $x\gamma =a$ for $x\in D(\alpha )$%
, and $x\gamma =x\alpha ^{-1}\beta $ for $x\in im\alpha $. Then
$x\alpha
\gamma =x\alpha \alpha ^{-1}\beta =x\beta $ for all $x\in X$. This shows $%
\alpha \gamma =\beta $, where $\gamma \notin Inj(X)$ since $D(\alpha
)\neq \emptyset $. This shows that $T(X)\setminus Inj(X)$ is
maximal.

2) Let $\alpha ,\beta \in T(X)\setminus Sur(X)$. Assume that $\alpha
\beta \in Sur(X)$. Then $d(\beta )=0$, i.e. $\beta $ is surjective.
Since $\beta \notin Sur(X)$, $\beta \in Sym(X)$. But $d(\alpha \beta
)=0$ and $\beta \in Sym(X)$ implies $\alpha $ is surjective. Since
$\alpha \notin Sur(X)$, $\alpha \in Sym(X)$. So $\alpha \beta \in
Sym(X)$, i.e. $\alpha \beta $
is injective, contradicts $\alpha \beta \in Sur(X)$. This shows that $%
T(X)\setminus Sur(X)$ is a semigroup.\newline Let $\alpha \in
Sur(X)$. Then we will show that $\left\langle T(X)\setminus
Sur(X),\alpha \right\rangle =T(X)$. For this let $\beta \in Sur(X)$.
For all
$\overline{x}\in X/\ker \alpha $ we fix a $\overline{x}^{\ast }\in \overline{%
x}$. Then we consider the following $\delta \in T(X)$ with $i\delta
=(i\beta \alpha ^{-1})^{\ast }$ for all $i\in X$. Hence $i\delta
\alpha =i\beta $ for
all $i\in X$. This shows $\delta \alpha =\beta $. Since $im\delta =\{%
\overline{x}^{\ast }\mid \overline{x}\in X/\ker \alpha \}\neq X$ (because $%
\alpha $ is not injective), $\delta \in T(X)\setminus Sur(X).$ This shows $%
\beta =\delta \alpha \in \left\langle T(X)\setminus Sur(X),\alpha
\right\rangle $. Consequently, $\left\langle T(X)\setminus
Sur(X),\alpha \right\rangle $ $=T(X)$. This shows that
$T(X)\setminus Sur(X)$ is maximal.

3) Let $\alpha ,\beta \in T(X)\setminus C_{p}(X)$. Further, let
$x\in im\alpha \beta $. Then $x(\alpha \beta )^{-1}=(x\beta
^{-1})\alpha ^{-1}$
and $\left\vert (x\beta ^{-1})\alpha ^{-1}\right\vert =\aleph _{0}$ if $%
\left\vert x\beta ^{-1}\right\vert =\aleph _{0}$ or there is a $y\in
x\beta
^{-1}\cap im\alpha $ with $\left\vert y\alpha ^{-1}\right\vert =\aleph _{0}$%
. This shows $k(\alpha \beta )\leq k(\alpha )+k(\beta )<\aleph
_{0}+\aleph
_{0}=\aleph _{0}$. This shows that $\alpha \beta \in T(X)\setminus C_{p}(X)$%
. \newline Let $\alpha \in C_{p}(X)$. Then, we will show that
$\left\langle T(X)\setminus C_{p}(X),\alpha \right\rangle =T(X)$.
For this let $\beta \in C_{p}(X)$. Then, there is a bijection
\[
f:X/\ker \beta \rightarrow \{x\alpha ^{-1}\mid x\in K(\alpha
)\}\text{.}
\]%
For each $\overline{x}\in X/\ker \beta $, there is an injective mapping%
\[
f_{\overline{x}}:\overline{x}\rightarrow f(\overline{x})\text{.}
\]%
We take the $\gamma \in Inj(X)$ with $i\gamma =f_{\overline{x}}(i)$ where $%
i\in \overline{x}$ for $\overline{x}\in X/\ker \beta $. For $i,j\in
X$, $i\beta =j\beta $ if and only if there is an $\overline{x}\in
X/\ker \beta $ with $i,j\in \overline{x}$, i.e. $f_{\overline{x}%
}(i)\alpha =f_{\overline{x}}(j)\alpha $. But $f_{\overline{x}}(i)\alpha =f_{%
\overline{x}}(j)\alpha $ is equivalent to $i\gamma \alpha =j\gamma
\alpha $,
consequently, we have $i\gamma \alpha =j\gamma \alpha $ if and only if $%
i\beta =j\beta $. Further, let $\delta \in T(X)$ with $i\gamma
\alpha \delta =i\beta $ for $i\in X$ and $i\delta =x_{0}$ ($x_{0}$
is any fixed element in $X$) for $i\in X\setminus im\gamma \alpha$.
Since $i\gamma \alpha =j\gamma \alpha $ if and only
if $i\beta =j\beta $, $\delta $ is well defined. Moreover, $\mid\{\overline{x}%
\mid \overline{x}\in X/\ker \delta ,\left\vert
\overline{x}\right\vert
=\aleph _{0}\}\mid <2$, i.e. $\delta \in T(X)\setminus C_{p}(X)$. This shows $%
\beta =\gamma \alpha \delta \in \left\langle T(X)\setminus
C_{p}(X),\alpha \right\rangle $. Consequently, $\left\langle
T(X)\setminus C_{p}(X),\alpha \right\rangle $ $=T(X)$. This shows
that $T(X)\setminus C_{p}(X)$ is maximal.

4) Let $\alpha ,\beta \in T(X)\setminus IF(X)$. \newline If
$c(\alpha )<\aleph _{0}$ and $c(\beta )<\aleph _{0}$ then $c(\alpha
\beta
)\leq c(\alpha )+c(\beta )<\aleph _{0}$, i.e. $\alpha \beta \notin IF(X)$.%
\newline
If $d(\alpha )=\aleph _{0}$ and $c(\beta )<\aleph _{0}$ then $\left\vert \{%
\overline{x}\in X/\ker \beta \mid \overline{x}\cap im\alpha
=\emptyset \}\right\vert =\aleph _{0}$. This implies $d(\alpha \beta
)=\aleph _{0}$, i.e. $\alpha \beta \notin IF(X)$.\newline
If $d(\beta )=\aleph _{0}$ then $d(\alpha \beta )\geq d(\beta )=\aleph _{0}$%
, i.e. $\alpha \beta \notin IF(X)$.\newline Altogether, this shows
that $\alpha \beta \in T(X)\setminus IF(X)$. \newline Let $\alpha
\in IF(X)$. Then we will show that $\left\langle T(X)\setminus
IF(X),\alpha \right\rangle =T(X)$. For this let $\beta \in IF(X)$. Let $%
\gamma \in T(X)$ with $\ker \gamma =\ker \beta $ and $im\gamma
\#\ker \alpha
$. For each $\overline{x}\in X/\ker \beta $, we fix any $\overline{x}%
^{\ast }$. Since $c(\alpha )=$ $\aleph _{0}$, $d(\gamma )=$ $\aleph
_{0}$, i.e. $\gamma \notin IF(X)$. Further, let $\delta \in T(X)$
with $im\alpha
\#\ker \delta $ and $i\delta =(i(\gamma \alpha )^{-1})^{\ast }\beta $ for $%
i\in im\alpha $. Since $im\gamma \#\ker \alpha $, we have $im\alpha
=im\gamma \alpha $ and $\delta $ is well defined. Because of
$im\gamma
\#\ker \alpha $, $\ker \gamma \alpha =\ker \gamma =\ker \beta $, where $%
i\gamma \alpha \delta =(i\gamma \alpha (\gamma \alpha )^{-1})^{\ast
}\beta =i\beta $ for $i\in X$. Note that from $im\alpha \#\ker
\delta $ and $d(\alpha )<\aleph _{0}$, it follows $c(\delta )<\aleph
_{0}$, i.e. $\delta \notin IF(X)$. This shows $\beta =\gamma \alpha
\delta \in \left\langle T(X)\setminus IF(X),\alpha \right\rangle $.
Consequently, $\left\langle T(X)\setminus IF(X),\alpha \right\rangle
$ $=T(X)$. This shows that $T(X)\setminus IF(X)$ is maximal.

5) Let $\alpha ,\beta \in T(X)\setminus FI(X)$. \newline If
$c(\alpha )=\aleph _{0}$ then $c(\alpha \beta )\geq c(\alpha
)=\aleph _{0} $, i.e. $\alpha \beta \notin FI(X)$.\newline If
$d(\alpha )<\aleph _{0}$ and $c(\beta )=\aleph _{0}$ then
$\left\vert
\{i\in im\alpha \mid \exists j\in im\alpha \setminus \{i\}\text{ such that }%
i\beta =j\beta\} \right\vert =\aleph _{0}$. This implies $c(\alpha
\beta )=\aleph _{0}$, i.e. $\alpha \beta \notin FI(X)$.\newline If
$d(\alpha )<\aleph _{0}$ and $d(\beta )<\aleph _{0}$\ then $d(\alpha
\beta )\leq d(\alpha )+d(\beta )<\aleph _{0}$, i.e. $\alpha \beta
\notin FI(X)$.\newline Altogether, this shows that $\alpha \beta \in
T(X)\setminus FI(X)$. \newline Let $\alpha \in FI(X)$. Then, we will
show that $\left\langle T(X)\setminus
FI(X),\alpha \right\rangle =T(X)$. For this, let $\beta \in FI(X)$. Let $%
\gamma \in T(X)$ with $\ker \gamma =\ker \beta $ and $im\gamma
\#\ker \alpha
$. For each $\overline{x}\in X/\ker \beta $, we fix any $\overline{x}%
^{\ast }$. Since $c(\alpha )<$ $\aleph _{0}$, $d(\gamma )<$ $\aleph
_{0}$, i.e. $\gamma \notin FI(X)$. Further, let $\delta \in T(X)$
with $im\alpha
\#\ker \delta $ and $i\delta =(i(\gamma \alpha )^{-1})^{\ast }\beta $ for $%
i\in im\alpha $. Since $im\gamma \#\ker \alpha $, $im\alpha
=im\gamma \alpha $ and $\delta $ is well defined, where $i\gamma
\alpha \delta =(i\gamma
\alpha (\gamma \alpha )^{-1})^{\ast }\beta =i\beta $ for $i\in X$. Note that from $%
im\alpha \#\ker \delta $ and $c(\alpha )=\aleph _{0}$, it follows
$d(\delta )=\aleph _{0}$, i.e. $\delta \notin FI(X)$. This shows
$\beta =\gamma \alpha \delta \in \left\langle T(X)\setminus
FI(X),\alpha \right\rangle $.
Consequently, $\left\langle T(X)\setminus FI(X),\alpha \right\rangle $ $%
=T(X) $. This shows that $T(X)\setminus FI(X)$ is maximal.
\end{proof}

First, we characterize the maximal subsemigroups of $T(X)$
containing $Inj(X)$ and $Sur(X)$, respectively. Note that we do not
need Theorem \ref{thm1} here. It is well known that the set $F(X)$
of all transformation of a finite rank forms an ideal of $T(X)$ and
$Inf(X):=T(X)\setminus F(X)$ generates $T(X)$. The next lemma shows
that any maximal subsemigroup $S$ of $T(X)$ has the form $S=F(X)\cup
T$ for some $T\subset Inf(X)$.

\begin{lemma}
Let $S$ be a maximal subsemigroup of $T(X)$. Then $F(X)\subset S$.
\label{lmm1}
\end{lemma}

\begin{proof}
We have $Inf(X)\nsubseteqq S$ (since $S\neq T(X)$). Since $F(X)$
forms an ideal of $T(X)$ and thus $ Inf(X)\nsubseteqq S$ implies
$S\subseteq S\cup F(X)\neq T(X)$. Because of the maximality of $S$,
we have $S=S\cup F(X)$, i.e. $F(X)\subset S$.
\end{proof}

\begin{lemma}
Let $Sur(X)\subset $ $S\leq $ $T(X)$ with $Inj(X)\cap S\neq \emptyset $ and $%
FI(X)\cap S\neq \emptyset $. Then
\[
S=T(X)\text{.}
\] \label{lmm2}
\end{lemma}

\begin{proof}
We have $F(X)\subset S$ by Lemma \ref{lmm1}. Hence, we have to
consider only the elements of $Inf(X)$. Let $\alpha \in Sym(X)$.
Then there is a $\beta \in
Inj(X)\cap S$ and we take the $\gamma \in Sur(X)$ with $i\gamma =i$ for $%
i\in D(\beta )$ and $i\beta \gamma =i\alpha $ for $i\in X$. Since
$im\beta =X\setminus D(\beta )$, this shows that $\beta \gamma
=\alpha $, and consequently, $Sym(X)\subset S$. Let us put
\begin{eqnarray*}
A &:&=\{\alpha \in Inf(X)\mid d(\alpha )<\aleph _{0}\} \\
B &:&=\{\alpha \in Inf(X)\mid d(\alpha )=\aleph _{0}\}\text{.}
\end{eqnarray*}%
Clearly, $Inf(X)=A\cup B$. Let $\alpha \in A$. If $d(\beta )<\aleph
_{0}$ then for each natural number $k\geq 1$, there is a natural
number $r\geq 1$ such that $d(\beta ^{r})\geq k$. Since $\beta
^{r}\in Inj(X)\cap S$, we can assume that $d(\beta )\geq d(\alpha
)$. Since $d(\beta )\geq d(\alpha )$, there is a $\gamma _{1}\in
Sur(X)$ such that $\gamma_1$ restricted to $im\beta$ is bijective
with $im\alpha$ as range and $ D(\beta )\gamma _{1}=D(\alpha )$. We
take the $\gamma _{2}\in Inf(X)$ with $ i\gamma _{2}$ is the unique
element in $i\alpha \gamma _{1}^{-1}\beta ^{-1}$ for $i\in X$. Since
$\beta $ is injective, we have $\gamma _{2}\in Sur(X)\cup Sym(X)$%
. Then we have $i\gamma _{2}\beta \gamma _{1}=i\alpha \gamma
_{1}^{-1}\beta ^{-1}\beta \gamma _{1}=i\alpha $ for $i\in X$. This
shows $\alpha =\gamma _{2}\beta \gamma _{1}\in S$, and consequently,
$A\subset S$.

Let $\alpha \in B$. Moreover, there is a $\delta \in FI(X)\cap S$.
Then there is a $\eta \in
A$ with $\ker \alpha =\ker \eta $ and $im\eta \#\ker \delta $. Since $%
d(\alpha )=d(\delta )=\aleph _{0}$, there is a bijection $f:D(\delta
)\rightarrow D(\alpha )$. We take the $\gamma _{3}\in Sym(X)$
defined by $i\gamma _{3}=f(i)$ for $i\in D(\delta )$ and $i\gamma
_{3}=i\delta ^{-1}\eta ^{-1}\alpha $ for $i\in im\delta $. Then for
$i\in X$, $i\eta \delta \gamma _{3}=(i\eta \delta \delta ^{-1}\eta
^{-1})\alpha =i\alpha $ since $im\eta \#\ker \delta $. This shows
$\alpha =\eta \delta \gamma _{3}$ and consequently, $B\subset S$.
Altogether, $Inf(X)=A\cup B\subseteq S$ and thus $S=T(X)$.
\end{proof}

\begin{lemma}
Let $Inj(X)\subset $ $S\leq $ $T(X)$ with $H\cap S\neq \emptyset $
for $H\in $ $\{Sur(X),C_{p}(X),IF(X)\}$. Then
\[
S=T(X)\text{.}
\]%
\label{lmm3}
\end{lemma}

\begin{proof}
We show that then $Sur(X)\subset S$. If we have $Sur(X)\subset S$ then from $%
Inj(X)\cap S\neq \emptyset $ and $FI(X)\cap S\neq \emptyset $ (because of $%
Inj(X)\subset $ $S$) it follow $S=T(X)$ by Lemma \ref{lmm2}.

Let $\alpha \in Sur(X)$. Moreover, there is a $\beta \in
C_{p}(X)\cap S$. Then there is a bijection
\[
f:X/\ker \alpha \rightarrow \{x\beta ^{-1}\mid x\in K(\beta
)\}\text{.}
\]%
For each $\overline{x}\in X/\ker \alpha $, there is an injective mapping%
\[
f_{\overline{x}}:\overline{x}\rightarrow f(\overline{x})\text{.}
\]%
We take the $\gamma \in Inj(X)$ with $i\gamma =f_{\overline{x}}(i)$ where $%
i\in \overline{x}$ for $\overline{x}\in X/\ker \alpha $. There are $\delta \in Sur(X)\cap S$ and $%
\eta \in IF(X)\cap S$. If $c(\delta )<\aleph _{0}$ then from
$c(\delta )>0$,
it follows $c(\delta ^{r})>d(\eta )$ for some $r\in \mathbb{N}$, where $%
\delta ^{r}\in Sur(X)\cap S$. Hence, we can assume that $c(\delta
)>d(\eta )$
and there is a set $A\subseteq X$ with $A\#\ker \delta $ and a bijection%
\[
h_{1}:im\eta \rightarrow A
\]%
and an injective mapping%
\[
h_{2}:D(\eta )\rightarrow X\setminus A\text{.}
\]%
We take the $\gamma _{1}\in Inj(X)$ with $i\gamma _{1}=h_{1}(i)$ for
$i\in im\eta $ and $i\gamma _{1}=h_{2}(i)$ for $i\in D(\eta )$.
Clearly, $\eta \gamma _{1}\delta \in Sur(X)\cap S$ with $c(\eta
\gamma _{1}\delta )=\aleph _{0}$.
So, we can assume that $c(\delta )=\aleph _{0}$. For $i,j\in X$, $%
i\alpha =j\alpha $ if and only if there is an $\overline{x}\in
X/\ker
\alpha $ with $i,j\in \overline{x}$, i.e. $f_{\overline{x}}(i)\beta =f_{%
\overline{x}}(j)\beta $. But $f_{\overline{x}}(i)\beta =f_{\overline{x}%
}(j)\beta $ is equivalent to $i\gamma \beta =j\gamma \beta $,
consequently, we have $i\gamma \beta =j\gamma \beta $ if and only if
$i\alpha =j\alpha $.
Further, let $B\subseteq X$ with $B\#\ker \delta $ and%
\[
\varphi :D(\gamma \beta )\rightarrow X\setminus B
\]%
be a injective mapping. Then the transformation $\gamma _{2}$ on $X$ with $%
i\gamma \beta \gamma _{2}$ is the unique element in $i\alpha \delta
^{-1}\cap B$ for $i\in X$ and $i\gamma _{2}=\varphi (i)$ for $i\in
D(\gamma \beta )$ belongs to $Inj(X)$. So, we have $i\gamma \beta
\gamma _{2}\delta =i\alpha \delta ^{-1}\delta =i\alpha $ for $i\in
X$. This shows that $\gamma \beta \gamma _{2}\delta =\alpha $, and
consequently, $Sur(X)\subset S$.
\end{proof}

Now we are able to characterize the maximal subsemigroups of $T(X)$
containing $Inj(X)$ and $Sur(X)$, respectively.

\begin{theorem}
Let $Sur(X)\subset $ $S\leq $ $T(X)$. Then $S$ is maximal iff $%
S=T(X)\setminus Inj(X)$ or $S=T(X)\setminus FI(X)$.
\end{theorem}

\begin{proof}
By Proposition \ref{prop1}, both $T(X)\setminus Inj(X)$ and
$T(X)\setminus FI(X)$ are maximal
subsemigroups of $T(X)$. Suppose that $S$ is a maximal subsemigroup of $T(X)$%
. Then $Inj(X)\cap S=\emptyset $ or $FI(X)\cap S=\emptyset $ by
Lemma \ref{lmm2}, i.e. $S\subseteq T(X)\setminus Inj(X)$ or
$S\subseteq T(X)\setminus FI(X)$ and thus $S=T(X)\setminus Inj(X)$
or $S=T(X)\setminus FI(X)$ because of the maximality of $S$.
\end{proof}

\begin{theorem}
Let $Inj(X)\subset $ $S\leq $ $T(X)$. Then $S$ is maximal iff $%
S=T(X)\setminus H$ for some $H\in \{Sur(X),C_{p}(X),IF(X)\}.$
\end{theorem}

\begin{proof}
By Proposition \ref{prop1}, $T(X)\setminus H$ ($H\in
\{Sur(X),C_{p}(X),IF(X)\}$) are maximal subsemigroups of $T(X)$. If
$S$ is a maximal subsemigroup of $T(X)$ then $H\cap S=\emptyset $
for some $H\in \{Sur(X),C_{p}(X),IF(X)\}$ by Lemma \ref{lmm3}, i.e.
$S\subseteq T(X)\setminus H$ for some $H\in
\{Sur(X),C_{p}(X),IF(X)\} $. The maximality of $S$ provides the
assertion.
\end{proof}

Finally, we want to determine the maximal subsemigroups of $T(X)$
containing $H$ for $H\in\{C_{p}(X),IF(X),FI(X)\}$ using Theorem
\ref{thm1}. First we state that $FI(X)$ as well as $IF(X)$ are
subsemigroups of $T(X)$.

\begin{lemma}
$FI(X)$ is a subsemigroup of $T(X)$. \label{lmm4}
\end{lemma}

\begin{proof}
Let $\alpha ,\beta \in FI(X)$. Then we have $c(\alpha \beta )\leq
c(\alpha )+c(\beta )<\aleph _{0}+\aleph _{0}=\aleph _{0}$ and
$\aleph _{0}=d(\beta )\leq d(\alpha \beta )$. This shows that
$\alpha \beta \in FI(X)$.
\end{proof}

\begin{lemma}
$IF(X)$ is a subsemigroup of $T(X)$. \label{lmm6}
\end{lemma}

\begin{proof}
Let $\alpha ,\beta \in IF(X)$. Then we have $d(\alpha \beta )\leq
d(\alpha )+d(\beta )<\aleph _{0}+\aleph _{0}=\aleph _{0}$ and
$\aleph _{0}=c(\beta )\leq c(\alpha \beta )$. This shows that
$\alpha \beta \in IF(X)$.
\end{proof}

Let us consider the set $C_{p}(X)\cap Sur(X)$. Then we have:

\begin{lemma}
We have $\left\langle FI(X),\alpha \right\rangle =T(X)$ for all
$\alpha \in C_{p}(X)\cap Sur(X)$. \label{lmm5}
\end{lemma}

\begin{proof}
Let $\alpha \in C_{p}(X)\cap Sur(X)$, $\beta \in Inj(X)$, and
$A\subseteq X$ be a transversal of $\alpha $. We put $\gamma \in
T(X)$ setting $x\gamma$ is the unique element in $x\beta \alpha
^{-1}\cap A$ for all $x\in X$. It is easy to verify that $im\gamma
\subseteq A$ and $d(\gamma )=\left\vert X\setminus im\gamma
\right\vert \geq \left\vert X\setminus A\right\vert
=\aleph _{0}$. Let $x,y\in X$ with $x\gamma =y\gamma $. This implies $%
(x\beta \alpha ^{-1}\cap A)\alpha =(y\beta \alpha ^{-1}\cap A)\alpha $, $%
x\beta =y\beta $, and $x=y$ since $\beta \in Inj(X)$. Thus $\gamma
\in Inj(X) $ and $c(\gamma )=0<\aleph _{0}$. Consequently, $\gamma
\in FI(X)$. Because of $x\gamma \alpha =(x\beta \alpha ^{-1}\cap
A)\alpha =x\beta $ for all $x\in X$, we have $\beta =\gamma \alpha
\in \left\langle FI(X),\alpha \right\rangle $. This shows that
$Inj(X)\subseteq \left\langle FI(X),\alpha \right\rangle $.
Moreover, $\left\langle FI(X),\alpha \right\rangle \cap
H\neq \emptyset $ for $H\in \{Sur(X),C_{p}(X),IF(X)\}$. By Lemma \ref{lmm3}, we have $%
\left\langle FI(X),\alpha \right\rangle =T(X)$.
\end{proof}

Lemma \ref{lmm4}, Lemma \ref{lmm5} and Theorem \ref{thm1} imply:

\begin{proposition}
Let $S\leq T(X)$ with $FI(X)\subseteq S$. Then the following
statements are equivalent:\newline (i) $S$ is maximal.\newline (ii)
$S=T(X)\setminus H$ for some $H\in \mathcal{H}(C_{p}(X)\cap
Sur(X),FI(X))$.
\end{proposition}

Now, we consider the set $FI(X)\cup Inj(X)$. Here, we get:

\begin{lemma}
We have $\left\langle IF(X),\alpha \right\rangle =T(X)$ for all
$\alpha \in FI(X)\cap Inj(X)$. \label{lmm7}
\end{lemma}

\begin{proof}
Let $\alpha \in FI(X)\cap Inj(X)$ and $\beta \in Inj(X)$. We put
$\gamma \in
T(X)$ setting%
\[
\begin{array}{c}
x\alpha \gamma :=x\beta \text{ for }x\in X \\
x\gamma :=f(x)\text{ for }x\in D(\alpha)
\end{array}%
\]%
where
\[f:D(\alpha) \rightarrow D(\beta)\cup x_{0}\alpha\]
is a surjective transformation such that
$|D(\alpha)\setminus\Sigma|=\aleph _{0}$ for some transversal
$\Sigma$ of $f$ and any fixed $x_{0}\in X$. Such a mapping exists
because of $d(\alpha )=\aleph _{0}$. Since $c(f)=\aleph _{0}$, we
have $c(\gamma )=\aleph _{0}$. Moreover, $im\gamma =\{x\gamma \mid
x\in X\}=\{x\gamma \mid x\in im\alpha \}\cup \{x\gamma \mid x\in
X\setminus im\alpha \}=im\beta \cup (X\setminus im\beta )\cup
\{x_{0}\alpha \}=X$. Hence $d(\gamma )=0$. This shows that $\gamma
\in IF(X)$. By definition, we have $\beta =\alpha \gamma \in
\left\langle IF(X),\alpha \right\rangle $. This shows that
$Inj(X)\subseteq \left\langle IF(X),\alpha \right\rangle $.
Moreover, $\left\langle IF(X),\alpha \right\rangle \cap H\neq
\emptyset $ for $H\in \{Sur(X),C_{p}(X),IF(X)\}$. By Lemma
\ref{lmm3}, we have $\left\langle IF(X),\alpha \right\rangle =T(X)$.
\end{proof}

Lemma \ref{lmm6}, Lemma \ref{lmm7} and Theorem \ref{thm1} imply:

\begin{proposition}
Let $S\leq T(X)$ with $IF(X)\subseteq S$. Then the following
statements are equivalent:\newline (i) $S$ is maximal.\newline (ii)
$S=T(X)\setminus H$ for some $H\in \mathcal{H}(Inj(X)\cap
FI(X),IF(X))$.
\end{proposition}

\begin{lemma}
$\left\langle C_{p}(X)\right\rangle \cap (Inj(X)\cap
FI(X))=\emptyset $.
\end{lemma}

\begin{proof}
Let $\alpha ,\beta \in T(X)$ with $c(\alpha )=c(\beta )=\aleph _{0}$. Then $%
\aleph _{0}=c(\alpha )\leq c(\alpha \beta )$, i.e. $c(\alpha \beta )=$ $%
\aleph _{0}$. Since $c(\alpha )=$ $\aleph _{0}$ for all $\alpha \in C_{p}(X)$%
, this shows that $\left\langle C_{p}(X)\right\rangle \cap
FI(X)=\emptyset $.
\end{proof}

Since we can decompose a countable set into countable many countable
sets, it is routine that each transformation $\alpha$ with $\exists
\overline{x}\in X/ker\alpha$ with $|\overline{x}|=\aleph_0$ can be
written as product $\beta\gamma$ of appropriate transformations
$\beta,\gamma\in C_p(X)$. Moreover, it is clear that $\{\alpha\in
T(X)|\exists \overline{x}\in X/ker\alpha$ with
$|\overline{x}|=\aleph_0\}$ is subsemigroup of $T(X)$. Hence
$\langle C_p(X)\rangle=\{\alpha\in T(X)|\exists \overline{x}\in
X/ker\alpha$ with $|\overline{x}|=\aleph_0\}$.

\begin{lemma}
We have $\left\langle C_{p}(X),\alpha \right\rangle =T(X)$ for all
$\alpha \in FI(X)\cap Inj(X)$.
\end{lemma}

\begin{proof}
We show that $Inj(X)\subset \left\langle C_{p}(X),\alpha
\right\rangle $. If we have it then from $\left\langle
C_{p}(X),\alpha \right\rangle \cap H\neq \emptyset $ for $H\in
\{Sur(X),C_{p}(X),IF(X)\}$ it follows $\left\langle C_{p}(X),\alpha
\right\rangle =T(X)$ by Lemma \ref{lmm3}. For this let $\beta \in
Inj(X)$. Let $\alpha \in S\cap (FI(X)\cap Inj(X))$. Further let
$\{I_{k}\mid k\in X\}$ be a decomposition of $D(\alpha )$ in
infinitely many infinite subsets. Then we take $\gamma \in
C_{p}(X)\subset \left\langle C_{p}(X),\alpha \right\rangle $ with
$X^{+}/\ker \gamma =\{I_{k}\cup \{k\alpha \}\mid k\in X\}$ and
$i\gamma =k\beta $ for $i\in I_{k}\cup
\{k\alpha \}$ and $k\in X$. This provides $k\alpha \gamma =k\beta $ for $%
k\in X$. This shows $\beta =\alpha \gamma \in S$. Consequently, $%
Inj(X)\subset \left\langle C_{p}(X),\alpha \right\rangle $.
\end{proof}

\begin{proposition}
Let $S\leq T(X)$ with $C_{p}(X)\subseteq S$. Then the following
statements are equivalent:\newline (i) $S$ is maximal.\newline (ii)
$S=T(X)\setminus H$ for some $H\in \mathcal{H}(FI(X)\cap
Inj(X),\left\langle C_{p}(X)\right\rangle )$.
\end{proposition}

\noindent{\bf Acknowledgments}

We gratefully acknowledge the financial support provided for this
research by Faculty of Science under grant $\sharp$ RGP 2554-10.


\begin{thebibliography}{9}


\bibitem{[3]} J. East, J. D. Mitchell, Y. Péresse,  Maximal Subsemigroups of
the Semigroup of all Mappings on an Infinite Set, arXiv:1104.2011V2

\bibitem{[4]} O. Ganyushkin and V. Mazorchuk, Classical Finite
Transformation Semigroups, Springer-Verlag, London, 2009.

\bibitem{[5]} L. Heindorf, The Maximal Clones on Countable Sets that include
all permutations, Algebra Universalis, 48(2) (2002), 209-222.

\bibitem{[6]} P.M. Higgins, J.M. Howie, and N. Ru\v{s}kuc, On Relative Ranks
of Full Transformation Semigroups, Comm. Algebra 26 (1998), 733-748.

\bibitem{[7]} J. M. Howie, \textit{Fundamentals of Semigroup Theory}, Oxford
University Press, 1995.

\bibitem{[8]} M. Pinsker, Maximal Clones on uncountable sets that include
all permutations, Algebra Universalis, 54(2) (2005), 129-148.

\end{thebibliography}
\end{document}